\documentclass{article}
\usepackage[initials]{amsrefs}

\usepackage{amsthm, amsmath, amssymb,graphicx}

\numberwithin{equation}{section}

\theoremstyle{plain}	     
\newtheorem{thm}{Theorem}[section] 

\newtheorem{lem}[thm]{Lemma}

\theoremstyle{definition}

\theoremstyle{remark} 
\newtheorem{rem}[thm]{Remark}
	
\newcommand{\sn}{\operatorname{sn}}
\newcommand{\cn}{\operatorname{cn}}
\newcommand{\dn}{\operatorname{dn}}

\newcommand{\sech}{\operatorname{sech}}

\begin{document}

\title{Applications of a duality between generalized trigonometric and hyperbolic functions
\footnote{The work of S.T. was supported by JSPS KAKENHI Grant Number 17K05336.}}
\author{Hiroki Miyakawa and Shingo Takeuchi \\
Department of Mathematical Sciences\\
Shibaura Institute of Technology
\thanks{307 Fukasaku, Minuma-ku,
Saitama-shi, Saitama 337-8570, Japan. \endgraf
{\it E-mail address\/}: shingo@shibaura-it.ac.jp \endgraf
}}


\maketitle

\begin{abstract}
It is shown that generalized trigonometric functions and generalized hyperbolic functions can be transformed from each other. As an application of this transformation, a number of properties for one immediately lead to the corresponding properties for the other. In this way, Mitrinovi\'{c}-Adamovi\'{c}-type inequalities, multiple-angle formulas, and double-angle formulas for both can be produced.
\end{abstract}



\section{Introduction}
\label{sec:1}

In almost all literature dealing with generalized trigonometric functions and generalized hyperbolic functions, e.g. \cite{Dosly2005,Klen,Neuman2015,YHWL2019}, both are currently studied independently. In this paper, by expanding the range of parameters contained in these functions beyond normal, we show that a duality between the two can be seen and one can be represented using the other, and we apply the duality to investigate formulas of generalized trigonometric and hyperbolic functions.

Before stating our results, we introduce the definitions and
some properties of generalized trigonometric and hyperbolic functions.

Let $q/(q+1)<p<\infty,\ 1<q<\infty$ and 
$$F_{p,q}(y):=\int_0^y \frac{dt}{(1-t^q)^{1/p}}, \quad y \in [0,1).$$
We will denote by $\sin_{p,q}$ the inverse function of $F_{p,q}$, i.e.,
$$\sin_{p,q}{x}:=F_{p,q}^{-1}(x).$$
Clearly, $\sin_{p,q}{x}$ is monotone increasing on
$[0,\pi_{p,q}/2)$ onto $[0,1)$,
where
$$\pi_{p,q}:=
2F_{p,q}(1)=2\int_0^1 \frac{dt}{(1-t^q)^{1/p}}=
\begin{cases}
(2/q)B(1/p^*,1/q), & 1<p<\infty,\\
\infty, & q/(q+1)<p \leq 1,
\end{cases}$$
and $p^*:=p/(p-1)$ and $B$ is the beta function.
In almost literature dealing with generalized trigonometric functions, 
the parameter $p$ is assumed to be $p>1$,
but we here allow it to be $p \leq 1$.
Note that the condition $q/(q+1)<p \leq 1$ implies that $\sin_{p,q}{x}$ is monotone
increasing on $[0,\infty)$ and no longer similar to $\sin{x}$, but to $\tanh{x}$.
Since $\sin_{p,q}{x} \in C^1(0,\pi_{p,q}/2)$,
we also define $\cos_{p,q}{x}$ by 
$$\cos_{p,q}{x}:=\frac{d}{dx}(\sin_{p,q}{x}).$$
Then, it follows that 
\begin{equation}
\label{eq:p+q=1}
\cos_{p,q}^p{x}+\sin_{p,q}^q{x}=1.
\end{equation}
In case $(p,q)=(2,2)$, it is obvious that $\sin_{p,q}{x},\ \cos_{p,q}{x}$ 
and $\pi_{p,q}$ are reduced to the ordinary $\sin{x},\ \cos{x}$ and $\pi$,
respectively. 
This is a reason why these functions and the constant are called
\textit{generalized trigonometric functions} (with parameter $(p,q)$)
and the \textit{generalized $\pi$}, respectively. 

The generalized trigonometric functions are well studied in the context of 
nonlinear differential equations (see \cite{Dosly2005,Lang2016,Kobayashi-Takeuchi,Takeuchi2016} and the references given there). 
Suppose that $u$ is a solution of 
the initial value problem of $p$-Laplacian
\begin{equation}
\label{eq:ivp}
(|u'|^{p-2}u')'+\frac{(p-1)q}{p} |u|^{q-2}u=0, \quad u(0)=0,\ u'(0)=1,
\end{equation}
which is reduced to the equation $u''+u=0$ of simple harmonic motion for
$u=\sin{x}$ in case $(p,q)=(2,2)$.  
Then, 
$$\frac{d}{dx}(|u'|^p+|u|^q)=\left(\frac{p}{p-1}(|u'|^{p-2}u')'+q|u|^{q-2}u\right)u'=0.$$
Therefore, $|u'|^p+|u|^q=1$, hence it is reasonable to define $u$ as 
a generalized sine function and $u'$ as a generalized cosine function.
Indeed, it is possible to show that $u$ coincides with $\sin_{p,q}$ defined as above.
The generalized trigonometric functions are often applied to
the eigenvalue problem of $p$-Laplacian.

In a similar way,
$$G_{p,q}(y):=\int_0^y \frac{dt}{(1+t^q)^{1/p}}, \quad y \in [0,\infty).$$
We will denote by $\sinh_{p,q}$ the inverse function of $G_{p,q}$, i.e.,
$$\sinh_{p,q}{x}:=G_{p,q}^{-1}(x).$$
Clearly, $\sinh_{p,q}{x}$ is monotone increasing
on $[0,\pi_{r,q}/2)$ onto $[0,\infty)$,
where $r$ is the constant determined by
\begin{equation}
\label{eq:r}
\frac{1}{p}+\frac{1}{r}=1+\frac{1}{q}, \quad \mbox{i.e.}, \quad
r=\frac{pq}{pq+p-q}.
\end{equation}
Indeed, by $1+t^q=1/(1-s^q)$,
\begin{align*}
\lim_{x \to \infty}G_{p,q}(x)=
\int_0^\infty \frac{dt}{(1+t^q)^{1/p}}
=\int_0^1 \frac{ds}{(1-s^q)^{1/r}}
=\frac{\pi_{r,q}}{2}.
\end{align*}
The important point to note here that for a fixed $q \in (1,\infty)$,
if $r=r(p)$ is regarded as a function of $p$, then 
\begin{gather}
p>\frac{q}{q+1}\ \Leftrightarrow\ r(p)>\frac{q}{q+1}, \label{eq:r(p)}\\
r(r(p))=p. \label{eq:r=p}
\end{gather}
In particular, $\pi_{r,q}$ has been defined when $p>q/(q+1)$.
Note that the condition $r>1$, i.e., $p<q$, implies that $\sinh_{p,q}{x}$ is defined in
the \textit{bounded} interval $[0,\pi_{r,q}/2)$ with 
$\lim_{x \to \pi_{r,q}/2}\sinh_{p,q}{x}=\infty$ and
no longer similar to $\sinh{x}$, but to $\tan{x}$.
Since $\sinh_{p,q}{x} \in C^1(0,\pi_{r,q}/2)$,
we also define $\cosh_{p,q}{x}$ by 
$$\cosh_{p,q}{x}:=\frac{d}{dx}(\sinh_{p,q}{x}).$$
Then, it follows that 
\begin{equation}
\label{eq:p-q=1}
\cosh_{p,q}^p{x}-\sinh_{p,q}^q{x}=1.
\end{equation}
In case $(p,q)=(2,2)$, it is obvious that $\sinh_{p,q}{x},\ \cosh_{p,q}{x}$ 
and the interval $[0,\pi_{r,q}/2)$ are reduced to 
$\sinh{x},\ \cosh{x}$ and $[0,\infty)$, respectively. 
This is a reason why these functions are called
\textit{generalized hyperbolic functions} (with parameter $(p,q)$). 
Just as $\sin_{p,q}{x}$ satisfies \eqref{eq:ivp}, 
$\sinh_{p,q}{x}$ is a solution of 
the initial value problem of $p$-Laplacian
$$(|u'|^{p-2}u')'-\frac{(p-1)q}{p} |u|^{q-2}u=0, \quad u(0)=0,\ u'(0)=1.$$

We are interested in finding a duality between 
generalized trigonometric and hyperbolic functions.

Let us start with a brief observation.
Since $F_{1,2}(x)=\tanh^{-1}{x}$, it is easily seen that
$$\sin_{1,2}{x}=\tanh{x}, \quad \cos_{1,2}{x}=\frac{1}{\cosh^2{x}}.$$
This yields that the hyperbolic functions can be written in terms of
the generalized trigonometric functions as
$$\sinh{x}=\frac{\sin_{1,2}{x}}{\sqrt{\cos_{1,2}{x}}}, 
\quad \cosh{x}=\frac{1}{\sqrt{\cos_{1,2}{x}}}.$$
From this observation, it is expected that any generalized hyperbolic function can be translated in terms of generalized trigonometric functions. In fact, 
we can prove the following crucial lemmas that play an important role in this paper.
\begin{lem}
\label{lem:sinhtosin}
Let $q/(q+1)<p<\infty,\ 1<q<\infty$ and $r$ be the number defined in \eqref{eq:r}. 
Then, for $x \in [0,\pi_{r,q}/2)$,
$$\sinh_{p,q}{x}=\frac{\sin_{r,q}{x}}{\cos_{r,q}^{r/q}{x}},\quad
\cosh_{p,q}{x}=\frac{1}{\cos_{r,q}^{r/p}{x}}.$$
\end{lem} 

\begin{lem}
\label{lem:sintosinh}
Let $q/(q+1)<p<\infty,\ 1<q<\infty$ and $r$ be the number defined in \eqref{eq:r}. 
Then, for $x \in [0,\pi_{p,q}/2)$,
$$\sin_{p,q}{x}=\frac{\sinh_{r,q}{x}}{\cosh_{r,q}^{r/q}{x}},\quad
\cos_{p,q}{x}=\frac{1}{\cosh_{r,q}^{r/p}{x}}.$$
\end{lem} 

Lemmas \ref{lem:sinhtosin} and \ref{lem:sintosinh} tell us the counterparts to generalized hyperbolic functions of the properties already known for generalized trigonometric functions, and vice versa.
For example, Lemma \ref{lem:sintosinh} immediately converts 
\eqref{eq:p+q=1} into \eqref{eq:p-q=1} (with $p$ replaced by $r$), that is,
$$\cos_{p,q}^p{x}+\sin_{p,q}^q{x}=1$$
into 
$$\cosh_{r,q}^r{x}-\sinh_{r,q}^q{x}=1;$$
and Lemma \ref{lem:sinhtosin} (with \eqref{eq:r=p}) does vice versa.

Using this idea with our lemmas, we will show the following generalization 
of Mitrinovi\'{c}-Adamovi\'{c} inequality. 

\begin{thm}
\label{thm:mai}
Let $q/(q+1)<p<\infty$ and $1<q<\infty$. Then, 
for $x \in (0,\pi_{p,q}/2)$,
\begin{equation}
\label{eq:mai}
\cos_{p,q}^{1/(q+1)}{x}<\frac{\sin_{p,q}{x}}{x}<1.
\end{equation}
Moreover, for $x \in (0,\pi_{r,q}/2)$, 
where $r$ is the number defined in \eqref{eq:r},
\begin{equation}
\label{eq:maih}
\cosh_{p,q}^{1/(q+1)}{x}<\frac{\sinh_{p,q}{x}}{x}<\cosh_{p,q}^{p/q}{x}.
\end{equation}
\end{thm}

Inequalities \eqref{eq:mai} and \eqref{eq:maih} were proved by Kl\'{e}n
et al. \cite[Theorems 3.6 and 3.8]{Klen}
when $p=q>1$; and Neuman \cite[Theorem 6.1 and (6.9)]{Neuman2015} 
when $p,\ q>1$ (with better upper bounds).  
In our approach, 
Theorem \ref{thm:mai} allows us to obtain the same inequalities 
over the wider range of parameters $q/(q+1)<p<\infty$ and $1<q<\infty$,
and \eqref{eq:maih} immediately follows from \eqref{eq:mai} 
by Lemma \ref{lem:sintosinh}.

As a further application of our lemmas, the following theorem is 
obtained by deriving the identities \eqref{eq:mafsinh} and \eqref{eq:mafcosh} below
of the generalized hyperbolic 
functions corresponding to the multiple-angle formulas 
\eqref{eq:mafsin} and \eqref{eq:mafcos} of the 
generalized trigonometric functions in \cite[Theorem 1.1]{Takeuchi2016}.

\begin{thm}
\label{thm:maf1}
Let $1<q<\infty$. Then, for $x \in [0,\pi_{2,q}/2^{2/q+1})=[0,\pi_{q^*,q}/4)$, 
\begin{gather}
\sin_{2,q}{(2^{2/q}x)}=2^{2/q}\sin_{q^*,q}{x}\cos_{q^*,q}^{q^*-1}{x}, \label{eq:mafsin}\\
\cos_{2,q}{(2^{2/q}x)}=\cos_{q^*,q}^{q^*}{x}-\sin_{q^*,q}^q{x}. \label{eq:mafcos}
\end{gather}
Moreover, for the same $x$,
\begin{gather}
\sinh_{2q/(q+2),q}{(2^{2/q}x)}=\frac{2^{2/q}\sinh_{q/2,q}{x}}{(1-\sinh_{q/2,q}^q{x})^{2/q}}, \label{eq:mafsinh} \\
\cosh_{2q/(q+2),q}{(2^{2/q}x)}
=\left(\frac{1+\sinh_{q/2,q}^q{x}}{1-\sinh_{q/2,q}^q{x}}\right)^{2/q+1}.
\label{eq:mafcosh}
\end{gather}
\end{thm}

\begin{rem}
In case of $q=2$, they are reduced to the double-angle formula
of $\sin,\ \cos,\ \tan$ and $\sec^2$.
\end{rem}

The multiple-angle formulas of the same type as \eqref{eq:mafsin} and \eqref{eq:mafcos} for generalized hyperbolic functions have not been found 
so far. In parallel with Theorem \ref{thm:maf1}, we 
establish them and furthermore deduce the corresponding formulas 
of the same type as \eqref{eq:mafsinh} and \eqref{eq:mafcosh}
for generalized trigonometric identities.

\begin{thm}
\label{thm:maf2}
Let $1<q<\infty$. Then, for $x \in [0,\pi_{2q/(q+2),q}/2^{2/q+1})=[0,\pi_{q/2,q}/2)$,
\begin{gather*} 
\sinh_{2,q}{(2^{2/q}x)}=2^{2/q}\sinh_{q^*,q}{x}\cosh_{q^*,q}^{q^*-1}{x},\\
\cosh_{2,q}{(2^{2/q}x)}=\cosh_{q^*,q}^{q^*}{x}+\sinh_{q^*,q}^q{x}.
\end{gather*}
Moreover, for the same $x$,
\begin{gather*}
\sin_{2q/(q+2),q}{(2^{2/q}x)}=\frac{2^{2/q}\sin_{q/2,q}{x}}{(1+\sin_{q/2,q}^q{x})^{2/q}},\\
\cos_{2q/(q+2),q}{(2^{2/q}x)}
=\left(\frac{1-\sin_{q/2,q}^q{x}}{1+\sin_{q/2,q}^q{x}}\right)^{2/q+1}.
\end{gather*}
\end{thm}

\begin{rem}
In case of $q=2$, they are reduced to the double-angle formula
of $\sinh,\ \cosh,\ \tanh$ and $\sech^2$.
\end{rem}

Double-angle formulas for generalized trigonometric functions,
along with the addition theorems, are of great interest, but few are available except for special parameters.
So far, only the double-angle formulas of $\sin_{q^*,2},\ \sin_{2,q}$ and $\sin_{q^*,q},\ 
q=2,3,4,6$, have been shown (Table \ref{tab:(p,q)}). 
For more details, see \cite{Takeuchipre} and the references given there.
Owing to Lemmas \ref{lem:sinhtosin}, \ref{lem:sintosinh}
and Theorems \ref{thm:maf1} and \ref{thm:maf2}, 
we succeed in finding the double-angle formulas of
$\sin_{3/2,6},\ \sin_{3,6}$ and $\sin_{6/5,3}$ (Table \ref{doubleanglehyo}), 
and the counterparts
for generalized hyperbolic functions (Table \ref{GTFGHFdobleanglehyo}
and Remark \ref{rem:th}). 
In our past papers we did not give the double-angle formula for 
$\sin_{3/2,2}$, but the present paper solves that too.


\begin{table}[!htb]
\centering
\begin{tabular}{|c|l|l|l|}
\noalign{\hrule height0.8pt}
$q$ & $(q^*,2)$ & $(2,q)$ & $(q^*,q)$  \\
\hline
$2$ & $(2,2)$ by Abu al-Wafa' & $(2,2)$ by Abu al-Wafa' & $(2,2)$ by Abu al-Wafa' \\
$3$ & $(3/2,2)$ \textbf{Theorem \ref{thm:3/2,2}} & $(2,3)$ by Cox-Shurman & $(3/2,3)$ by Dixon \\
$4$ & $(4/3,2)$ by Sato-Takeuchi & $(2,4)$ by Fagnano & $(4/3,4)$ by Edmunds et al. \\
$6$ & $(6/5,2)$ by Takeuchi & $(2,6)$ by Shinohara & $(6/5,6)$ by Takeuchi \\
\noalign{\hrule height0.8pt}
\end{tabular}
\caption{The parameters for which the double-angle formulas of GTF have been obtained.}
\label{tab:(p,q)}
\end{table}

\begin{table}[!htb]
\centering
  \begin{tabular}{|c|l|l|l|} \hline
    $q$ & $(q/2,q)$ & $(2q/(2+q),q)$ & $(2q/(2+q),2)$ \\ \hline 
    $2$ & $(1,2)$ by V. Riccati & $(1,2)$ by V. Riccati & $(1,2)$ by V. Riccati\\
    $3$ & $(3/2,3)$ by Dixon & $(6/5,3)$ \textbf{Theorem \ref{lem:6/5,3}} & $(6/5,2)$ by Takeuchi\\
    $4$ & $(2,4)$ by Fagnano & $(4/3,4)$ by Edmunds et al. & $(4/3,2)$ by Sato-Takeuchi\\ 
    $6$ & $(3,6)$ \textbf{Theorem \ref{lem:3,6}} & $(3/2,6)$ \textbf{Theorem \ref{lem:3/2,6}} & $(3/2,2)$ \textbf{Theorem \ref{thm:3/2,2}}\\\hline
\end{tabular}
\caption{The parameters for which the double-angle formulas of GTF have been obtained.}
\label{doubleanglehyo}
\end{table}

\begin{table}[!htb]
\centering
  \begin{tabular}{|c|c|c|} \hline
    GTF & $\longleftrightarrow$ & GHF  \\ \hline 
    $(q^*,2)$ &  & $(2q/(2+q),2)$ \\
    $(2,q)$ &  & $(2q/(2+q),q)$ \\
    $(q^*,q)$ & $\longleftrightarrow$ & $(q/2,q)$ \\ 
    $(q/2,q)$ & & $(q^*,q)$ \\
    $(2q/(2+q),q)$ & & $(2,q)$ \\
    $(2q/(2+q),2)$ & & $(q^*,2)$ \\\hline
\end{tabular}
\caption{The parameters that can be converted in 
Lemmas \ref{lem:sinhtosin} and \ref{lem:sintosinh}.}
\label{GTFGHFdobleanglehyo}
\end{table}

This paper is organized as follows.
In Section \ref{sec:2}, we will prove the lemmas and theorems introduced in Section \ref{sec:1}. Section \ref{sec:3} is devoted to using these results to produce new double-angle formulas for generalized trigonometric and hyperbolic functions. Finally, Section \ref{sec:4} gives some notes on the generalization of the tangent function.


\section{Proofs of Results}
\label{sec:2}

In this section, we will prove the lemmas and theorems introduced in Section 1. 
Properties \eqref{eq:r(p)} and \eqref{eq:r=p} are used frequently.

\begin{proof}[Proof of Lemma \ref{lem:sinhtosin}]
Let $p>q/(q+1)$. 
Then, $r>q/(q+1)$ by \eqref{eq:r(p)} and
the integration
$$\sin_{r,q}^{-1}{y}=\int_0^y \frac{dt}{(1-t^q)^{1/r}}, \quad y \in [0,1),$$
with $1-t^q=1/(1+s^q)$ gives
\begin{align*}
\sin_{r,q}^{-1}{y}
=\int_0^{y/(1-y^q)^{1/q}} \frac{ds}{(1+s^q)^{1/p}}
=\sinh_{p,q}^{-1}{\left(\frac{y}{(1-y^q)^{1/q}}\right)}.
\end{align*}
Hence, we obtain
$$\sinh_{p,q}{x}=\frac{\sin_{r,q}{x}}{\cos_{r,q}^{r/q}{x}}, \quad
x \in \left[0, \frac{\pi_{r,q}}{2}\right).$$
Also, by \eqref{eq:p-q=1},
$$\cosh_{p,q}{x}=\frac{1}{\cos_{r,q}^{r/p}{x}},$$
and the proof is complete.
\end{proof}

\begin{proof}[Proof of Lemma \ref{lem:sintosinh}]
Let $p>q/(q+1)$.
Then, $r>q/(q+1)$ by \eqref{eq:r(p)}.
Lemma \ref{lem:sinhtosin} with $p$ replaced by $r$ and \eqref{eq:r=p} show
$$\sinh_{r,q}{x}=\frac{\sin_{p,q}{x}}{\cos_{p,q}^{p/q}{x}}, \quad
\cosh_{r,q}{x}=\frac{1}{\cos_{p,q}^{p/r}{x}}, \quad x \in \left[0,\frac{\pi_{p,q}}{2}\right).$$
Therefore,
$$\sin_{p,q}{x}=\frac{\sinh_{r,q}{x}}{\cosh_{r,q}^{r/q}{x}},\quad
\cos_{p,q}{x}=\frac{1}{\cosh_{r,q}^{r/p}{x}},$$
and the proof is complete.
\end{proof}

\begin{proof}[Proof of Theorem \ref{thm:mai}]
We prove \eqref{eq:mai}. 
Since $\sin_{p,q}{x}<x$ in $(0,\pi_{p,q}/2)$, the second inequality holds; 
hence we show the first inequality.

Let $f(x):=x-\sin_{p,q}{x}\cos_{p,q}^{-1/(q+1)}{x}$. Then,
$$f'(x)=1-\cos_{p,q}^{q/(q+1)}{x}-\frac{q}{p(q+1)}\sin_{p,q}^q{x}\cos_{p,q}^{q/(q+1)-p}{x},$$
$$f''(x)=-\frac{q^2(pq+p-q)}{p^2(q+1)^2}
\sin_{p,q}^{2q-1}{x}\cos_{p,q}^{(2q+1)/(q+1)-2p}{x}.$$
Since $pq+p-q>0$, we see $f''(x)<0$.
Therefore, $f'(x)<\lim_{x \to +0}f'(x)=0$.
Moreover, $f(x)<\lim_{x \to +0}f(x)=0$, which means \eqref{eq:mai}.

Next we show \eqref{eq:maih}.
Let $p>q/(q+1)$.
Then, $r>q/(q+1)$ by \eqref{eq:r(p)}.
From \eqref{eq:mai} with $p$ replaced by $r$,  
$$\cos_{r,q}^{1/(q+1)}{x}<\frac{\sin_{r,q}{x}}{x}<1.$$
By Lemma \ref{lem:sintosinh} with $p$ replaced by $r$ and \eqref{eq:r=p}, 
$$\left(\frac{1}{\cosh_{p,q}^{p/r}{x}}\right)^{1/(q+1)}
<\frac{\sinh_{p,q}{x}}{x \cosh_{p,q}^{p/q}{x}}<1.$$
Multiplying both sides by $\cosh_{p,q}^{p/q}{x}$, we conclude \eqref{eq:maih}.
\end{proof}

\begin{proof}[Proof of Theorem \ref{thm:maf1}]
The former half is the multiple-angle formulas of generalized trigonometric 
functions, which were proved in \cite{Takeuchi2016}. 

The latter half is shown as follows. 
By Lemma \ref{lem:sinhtosin} with $r(2q/(q+2))=2$ and the former half, 
$$\sinh_{2q/(q+2),q}{(2^{2/q}x)}
=\frac{\sin_{2,q}{(2^{2/q}x)}}{\cos_{2,q}^{2/q}{(2^{2/q}x)}}
=\frac{2^{2/q}\sin_{q^*,q}{x}\cos_{q^*,q}^{q^*-1}{x}}{(\cos_{q^*,q}^{q^*}{x}-\sin_{q^*,q}^q{x})^{2/q}}.$$
Lemma \ref{lem:sintosinh} with $r(q^*)=q/2$ shows 
that the right-hand side becomes
$$
\frac{2^{2/q}\sinh_{q/2,q}{x}}{(1-\sinh_{q/2,q}^q{x})^{2/q}}.$$
The formula of $\cosh_{2q/(q+2),q}$ immediately follows from  
\eqref{eq:p-q=1}.
\end{proof}

\begin{proof}[Proof of Theorem \ref{thm:maf2}]
The former half is shown as follows.
Let $y \in [0,\infty)$. Setting $t^q=((1+s^q)^{1/2}-1)/2$ in
$$\sinh_{q^*,q}^{-1}{y}=\int_0^y \frac{dt}{(1+t^q)^{1/q^*}},$$
we have
\begin{align*}
\sinh_{q^*,q}^{-1}{y}
&=\int_0^{y(4(1+y^q))^{1/q}}
\dfrac{\dfrac{2^{-1-1/q}s^{q-1}}{(1+s^q)^{1/2}((1+s^q)^{1/2}-1)^{1-1/q}}\,ds}
{2^{-1+1/q} ((1+s^q)^{1/2}+1)^{1-1/q}}\\
&=2^{-2/q}
\int_0^{y(4(1+y^q))^{1/q}} 
\frac{ds}{(1+s^q)^{1/2}};
\end{align*}
that is, 
\begin{equation}
\label{eq:pe}
\sinh_{q^*,q}^{-1}{y}=2^{-2/q} \sinh_{2,q}^{-1}{(y(4(1+y^q))^{1/q})}.
\end{equation}
Hence we obtain
$$\sinh_{2,q}{(2^{2/q} x)}=2^{2/q} \sinh_{q^*,q}{x}\cosh_{q^*,q}^{q^*-1}{x}.$$
In particular, letting $y \to \infty$ in \eqref{eq:pe} and using $r(q^*)=q/2,\
r(2)=2q/(q+2)$, 
we get
$$\frac{\pi_{q/2,q}}{2}=\frac{\pi_{2q/(q+2),q}}{2^{1+2/q}}.$$
The formula of $\cosh_{q^*,q}$ immediately follows from  
\eqref{eq:p-q=1}.

The latter half is proved as follows.
By Lemma \ref{lem:sintosinh} with $r(2q/(q+2))=2$ and the former half, 
$$\sin_{2q/(q+2),q}{(2^{2/q}x)}
=\frac{\sinh_{2,q}{(2^{2/q}x)}}{\cosh_{2,q}^{2/q}{(2^{2/q}x)}}
=\frac{2^{2/q}\sinh_{q^*,q}{x}\cosh_{q^*,q}^{q^*-1}{x}}{(\cosh_{q^*,q}^{q^*}{x}+\sinh_{q^*,q}^q{x})^{2/q}}.$$
Lemma \ref{lem:sinhtosin} with $r(q^*)=q/2$ shows
that the right-hand side becomes
$$
\frac{2^{2/q}\sin_{q/2,q}{x}}{(1+\sin_{q/2,q}^q{x})^{2/q}}.$$
The formula of $\cos_{2q/(q+2),q}$ immediately follows from  
\eqref{eq:p+q=1}.
\end{proof}


\section{Double-angle formulas}
\label{sec:3}

We first show the double-angle formula of $\sinh_{2,6}$. 
The proof is almost the same for $\sin_{2,6}$ in \cite{Takeuchipre}, 
but it is included here for the convenience of the reader. 
Next, based on this, we prove the double-angle formula of $\sin_{3/2,6}$
from that of $\sinh_{2,6}$ by Lemma \ref{lem:sintosinh}; 
the formula of $\sin_{3,6}$ (resp. $\sin_{6/5,3}$) 
from that of $\sin_{3/2,6}$ (resp. $\sin_{3/2,3}$) 
by Theorem \ref{thm:maf2}.

\begin{lem}
\label{lem:2,6}
For $x \in [0,\pi_{3/2,6}/4)$,
\begin{equation}
\label{eq:2,6}
\sinh_{2,6}(2x)=\frac{2\sinh_{2,6}x\cosh_{2,6}x}{\sqrt{1-8\sinh_{2,6}^6x}},
\end{equation}
$$\cosh_{2,6}(2x)=\frac{1+20\sinh_{2,6}^6x-8\sinh_{2,6}^{12}x}{(1-8\sinh_{2,6}^6x)^{3/2}}.$$
\end{lem}
\begin{proof}
The change of variable $s=t^2$ in
$$\sinh_{2,6}^{-1}y=\int^y_0\frac{dt}{\sqrt{1+t^6}}, \quad y \in [0,\infty),$$
leads to the representation
$$\sinh_{2,6}^{-1}y=\frac{1}{2}\int^{y^2}_0\frac{ds}{\sqrt{s(1+s^3)}}.$$
The furthermore change of variable 
$$\cn u=\frac{1-(\sqrt{3}-1)s}{1+(\sqrt{3}+1)s},\quad k^2=\frac{2+\sqrt{3}}{4},$$
gives
\begin{align*}
\sinh_{2,6}^{-1}y&=\frac{1}{2}\int^{\cn^{-1}\phi(y)}_0\frac{((\sqrt{3}+1)\cn u+(\sqrt{3}-1))^2}{2\cdot3^{3/4}\sn u\dn u}\frac{2\sqrt{3}\sn u\dn u}{((\sqrt{3}+1)\cn u+(\sqrt{3}-1))^2}du\\
&=\frac{1}{2\cdot3^{1/4}}\cn^{-1}\phi(x),
\end{align*}
where $\sn{u}=\sn{(u,k)},\ \cn{u}=\cn{(u,k)}$ and $\dn{u}=\dn{(u,k)}$ are the
Jacobian elliptic functions, and
$$\phi(y)=\frac{1-(\sqrt{3}-1)y^2}{1+(\sqrt{3}+1)y^2}.$$
Thus, 
$$\sinh_{2,6}x=\phi^{-1}(\cn(2\cdot3^{1/4}x)),\quad
x \in \left[0, \frac{\pi_{3/2,6}}{2}\right),$$
and
$$\phi^{-1}(x)=\sqrt{\frac{1-x}{(\sqrt{3}-1)+(\sqrt{3}+1)x}}.$$
For $0\le x<\pi_{3/2,6}/4$, 
\begin{align*}
\sinh_{2,6}(2x)=&\phi^{-1}(\cn(2\tilde{x}))\\
=&\phi^{-1}\left(\frac{\cn^2\tilde{x}-\sn^2\tilde{x}\dn^2\tilde{x}}{1-k^2\sn^4\tilde{x}}\right)
\end{align*}
where $\tilde{x}:=2\cdot3^{1/3}x$. Recall that $\sn^2z=1-\cn^2z,\ 
\dn^2z=1-k^2\sn^2{z}$; then the last equality gives
\begin{align*}
\sinh_{2,6}(2x)
=\phi^{-1}\left(\frac{\phi(X)^2-(1-\phi(X)^2)(1-k^2(1-\phi(X)^2)}{1-k^2(1-\phi(X)^2)^2}\right),
\end{align*}
where $X:=\sinh_{2,6}x$.
With the observation 
\begin{gather*}
1-\phi(X)^2
=\frac{4\sqrt{3}X^2(1+X^2)}{(1+(\sqrt{3}+1)X^2)^2},\\
1-k^2(1-\phi(X)^2)
=\frac{1-X^2+X^4}{(1+(\sqrt{3}+1)X^2)^2},\\
1-k^2(1-\phi(X)^2)^2
=\frac{1+4(\sqrt{3}+1)X^2-8X^6+4(\sqrt{3}+1)X^8}{(1+(\sqrt{3}+1)X^2)^4},
\end{gather*}
implies that
\begin{align*}
\sinh_{2,6}(2x)=&\phi^{-1}\left(\frac{\phi(X)^2-(1-\phi(X)^2)(1-k^2(1-\phi(X)^2)}{1-k^2(1-\phi(X)^2)^2}\right)\\
=&\phi^{-1}\left(\frac{1-4(\sqrt{3}-1)X^2-8X^6-4(\sqrt{3}-1)X^8}{1+4(\sqrt{3}+1)X^2-8X^6+4(\sqrt{3}+1)X^8}\right).
\end{align*}
Routine simplification now results in the formula
\begin{equation*}
\sinh_{2,6}(2x)=\frac{2X\sqrt{1+X^6}}{\sqrt{1-8X^6}}
=\frac{2\sinh_{2,6}x\cosh_{2,6}x}{\sqrt{1-8\sinh_{2,6}^6x}}.
\end{equation*}
The formula of $\cosh_{2,6}$ immediately follows from  
\eqref{eq:p-q=1}.
\end{proof}

\begin{rem}
Since the left-hand side of \eqref{eq:2,6}
diverges to $\infty$ as $x \to \pi_{3/2,6}/4-0$,
we obtain 
$$\sinh_{2,6}{\frac{\pi_{3/2,6}}{4}}=\frac{1}{\sqrt{2}}.$$ 
\end{rem}

\begin{thm}\label{lem:3/2,6}
For $x\in [0,\pi_{3/2,6}/4)$,
\begin{equation}
\label{eq:3/2,6}
\sin_{3/2,6}(2x)=\frac{2\sin_{3/2,6}x}{\left(1+18\sin_{3/2,6}^6x-27\sin_{3/2,6}^{12}x\right)^{1/3}}.
\end{equation}
\end{thm}
\begin{proof}
By Lemma \ref{lem:sintosinh} with $p=3/2$ and $q=6$,
for $x \in [0,\pi_{3/2,6}/4)$,
$$\sin_{3/2,6}{(2x)}=\frac{\sinh_{2,6}{(2x)}}{\cosh_{2,6}^{1/3}{(2x)}}.$$
Applying Lemma \ref{lem:2,6} to the right-hand side gives
$$\sin_{3/2,6}{(2x)}=\frac{2\sinh_{2,6}{x}\cosh_{2,6}{x}}{(1+20\sinh_{2,6}^6{x}-8\sinh_{2,6}^{12}{x})^{1/3}}.$$
Lemma \ref{lem:sinhtosin} with $p=2$ and $q=6$ yields
$$\sin_{3/2,6}{(2x)}=\frac{2\sin_{3/2,6}{x}}{(\cos_{3/2,6}^3{x}+20\sin_{3/2,6}^6{x}
\cos_{3/2,6}^{3/2}{x}-8\sin_{3/2,6}^{12}{x})^{1/3}}.$$
The conclusion follows from \eqref{eq:p+q=1}.
\end{proof}

\begin{rem}
Since the left-hand side of \eqref{eq:3/2,6}
converges to $1$ as $x \to \pi_{3/2,6}/4-0$,
we obtain 
$$\sin_{3/2,6}{\frac{\pi_{3/2,6}}{4}}=\frac{1}{3^{1/3}}.$$ 
\end{rem}

\begin{rem}
\label{rem:th}
As the proof of Theorem \ref{lem:3/2,6},
we can also obtain 
a formula of $\sinh_{p,q}$ from those of $\sin_{r,q}$ and $\cos_{r,q}$.
For instance, combining Lemmas \ref{lem:sinhtosin}, \ref{lem:sintosinh}
and the formula of $\sin_{4/3,4}$ due to Edmunds et al. \cite{Edmunds2012}:
$$\sin_{4/3,4}{(2x)}=\frac{2\sin_{4/3,4}{x}\cos_{4/3,4}^{1/3}{x}}
{\sqrt{1+4\sin_{4/3,4}^4{x}\cos_{4/3,4}^{4/3}{x}}},$$
we have
$$\sinh_{2,4}{(2x)}=\frac{2\sinh_{2,4}x \cosh_{2,4}x}{1-\sinh_{2,4}^4x},$$
which is very similar to the double-angle formula of lemniscate function: 
$$\sin_{2,4}{(2x)}=\frac{2\sin_{2,4}x \cos_{2,4}x}{1+\sin_{2,4}^4x}.$$
Thus, formulas for generalized hyperbolic
functions can be produced from those for generalized trigonometric
functions, and vice versa.
\end{rem}

\begin{thm}\label{lem:3,6}
For $x \in [0,\pi_{3,6}/4)$,
\begin{equation}
\label{eq:3,6}
\sin_{3,6}(2x)=\frac{2^{5/3}s(1+s^6)}{\left((1-s^3)(1+6s^3+s^6)^{3/2}+(1+s^3)(1-6s^3+s^6)^{3/2}\right)^{2/3}},
\end{equation}
where $s:=\sin_{3,6}x$.
\end{thm}

\begin{proof}
Let $x \in (0,\pi_{3,6}/4)$, since the theorem is trivial if $x=0$.  
Applying Theorem \ref{thm:maf2} for $q=6$
with $x$ replaced by $2x \in (0,\pi_{3/2,6}/2^{4/3})=(0,\pi_{3,6}/2)$, we get
$$\sin_{3/2,6}(2 \cdot 2^{1/3}x)=\frac{2^{1/3}\sin_{3,6}{(2x)}}{\left(1+\sin_{3,6}^6{(2x)}\right)^{1/3}};$$
that is,
$$\sin_{3/2,6}^3(2\cdot 2^{1/3}x)\sin_{3,6}^6{(2x)}-2\sin_{3,6}^3{(2x)}+\sin_{3/2,6}^3(2 \cdot 2^{1/3}x)=0.$$
Therefore, since $0<\sin_{3,6}^3{(2x)}\sin_{3/2,6}^3(2\cdot 2^{1/3}x)\le1$,
$$\sin_{3,6}^3{(2x)}=\frac{1-\sqrt{1-\sin_{3/2,6}^6(2 \cdot 2^{1/3}x)}}{\sin_{3/2,6}^3(2 \cdot 2^{1/3}x)}.$$
Set $S=S(x):=\sin_{3/2,6}(2^{1/3}x)$. Using Theorem \ref{lem:3/2,6} for $\sin_{3/2,6}{x}$, we have
\begin{align*}
\sin_{3,6}^3(2x)
&=\frac{1-\sqrt{1-64S^6/(1+18S^6-27S^{12})^2}}{8S^3/(1+18S^6-27S^{12})}.
\end{align*}
For $a>b>0$, the identity
$$\frac{1-\sqrt{1-b^2/a^2}}{b/a}=\frac{2b}{(\sqrt{a+b}+\sqrt{a-b})^2}$$
holds; hence
\begin{align*}
\sin_{3,6}^3(2x)
&=\frac{16S^3}{(\sqrt{1+8S^3+18S^6-27S^{12}}+\sqrt{1-8S^3+18S^6-27S^{12}})^2}\\
&=\frac{16S^3}{(\sqrt{(1-S^3)(1+3S^3)^3}+\sqrt{(1+S^3)(1-3S^3)^3})^2}.
\end{align*}
Here, by Theorem \ref{thm:maf2} with $q=6$, $s=s(x):=\sin_{3,6}x$ satisfies
$$S^3=\frac{2s^3}{1+s^6}.$$
Thus, 
\begin{align*}
\sin_{3,6}^3(2x)
&=\frac{32s^3(1+s^6)^{3}}{(\sqrt{(1-s^3)^2(1+6s^3+s^6)^3}+\sqrt{(1+s^3)^2(1-6s^3+s^6)^3})^2},
\end{align*}
which establishes the formula.
\end{proof}

\begin{rem}
Since the left-hand side of \eqref{eq:3,6}
converges to $1$ as $x \to \pi_{3,6}/4-0$,
we obtain 
$$\sin_{3,6}{\frac{\pi_{3,6}}{4}}=(3-2\sqrt{2})^{1/3}.$$ 
\end{rem}

\begin{thm}\label{lem:6/5,3}
For $x\in[0,\pi_{6/5,3}/4)$,
\begin{equation}
\label{eq:6/5,3}
\sin_{6/5,3}(2x)=\frac{4\cos_{6/5,3}^{1/5}x(1+3\cos_{6/5,3}^{3/5}x)(1-\cos_{6/5,3}^{3/5}x)^{1/3}}{(1+24\cos_{6/5,3}^{3/5}x+18\cos_{6/5,3}^{6/5}x-27\cos_{6/5,3}^{12/5}x)^{2/3}}.
\end{equation}
\end{thm}

\begin{proof}
From Theorem \ref{thm:maf2} with $q=3$, 
for $x \in [0,\pi_{6/5,3}/2^{5/3})=[0,\pi_{3/2,3}/2)$,
$$\sin_{6/5,3}(2^{2/3}x)=\frac{2^{2/3}\sin_{3/2,3}x}{\left(1+\sin_{3/2,3}^3x\right)^{{2/3}}}.$$
In a similar way to the proof of Theorem \ref{lem:3,6}, we have
\begin{equation}
\label{eq:trans}
\sin_{3/2,3}^{3/2}x
=\sqrt{\frac{1-\cos_{6/5,3}^{3/5}(2^{2/3}x)}{1+\cos_{6/5,3}^{3/5}(2^{2/3}x)}}.
\end{equation}

Now, let $x \in [0,\pi_{6/5,3}/4)$ and $y:=x/(2^{2/3})$.
It follows from Theorem \ref{thm:maf2} with $q=3$ that
since $2y\in[0,\pi_{6/5,3}/2^{5/3})=[0,\pi_{3/2,3}/2)$, we get
\begin{equation}
\label{eq:conc}
\sin_{6/5,3}(2x)=\sin_{6/5,3}(2^{2/3}\cdot2y)
=\frac{2^{2/3}\sin_{3/2,3}(2y)}{\left(1+\sin_{3/2,3}^3(2y)\right)^{2/3}}.
\end{equation}
Here, Dixon's formula (Lemma \ref{lem:dixon}) and \eqref{eq:trans} yield
\begin{align*}
\sin_{3/2,3}(2y)
&=\frac{\sin_{3/2,3}y\ (2-\sin_{3/2,3}^3y)}{(1-\sin_{3/2,3}^3y)^{1/3}(1+\sin_{3/2,3}^3y)}\\
&=\frac{(1+3c^{3/5})(1-c^{3/5})^{1/3}}{2^{4/3}c^{1/5}},
\end{align*}
where $c:=\cos_{6/5,3}(2^{2/3}y)=\cos_{6/5,3}{x}$.
Therefore, from \eqref{eq:conc} we have
\begin{align*}
\sin_{6/5,3}(2x)
&=\frac{4c^{1/5}(1+3c^{3/5})(1-c^{3/5})^{1/3}}{(1+24c^{3/5}+18c^{6/5}-27c^{12/5})^{2/3}}.
\end{align*}
The proof is completed.
\end{proof}

\begin{rem}
Since the left-hand side of \eqref{eq:6/5,3}
converges to $1$ as $x \to \pi_{6/5,3}/4-0$,
we obtain 
$$\cos_{6/5,3}{\frac{\pi_{6/5,3}}{4}}=\frac{1}{3^{5/3}}.$$ 
\end{rem}

No double-angle formula for $\sin_{3/2,2}$ has been resolved so far.
Finally in this section, we prove the previously unresolved 
double-angle formula of $\sin_{3/2,2}$, though  
this formula is not an application of Lemmas \ref{lem:sinhtosin} and 
\ref{lem:sintosinh}. 
Here are two lemmas needed to prove the formula.

\begin{lem}[Dixon \cite{Dixon1890}]
\label{lem:dixon}
Let $u+v,\ u,\ v\in[0,\pi_{3/2,3}/2)$ with $u \neq v$. Then,
\begin{align*}
\sin_{3/2,3}(u+v)
&=\frac{\sin_{3/2,3}^2u\cos^{1/2}_{3/2,3}v-\cos^{1/2}_{3/2,3}u\sin_{3/2,3}^2v}{\sin_{3/2,3}u\cos_{3/2,3}v-\cos_{3/2,3}u\sin_{3/2,3}v},\\
\cos^{1/2}_{3/2,3}(u+v)&=\frac{\sin_{3/2,3}u\cos^{1/2}_{3/2,3}u-\cos^{1/2}_{3/2,3}v\sin_{3/2,3}v}{\sin_{3/2,3}u\cos_{3/2,3}v-\cos_{3/2,3}u\sin_{3/2,3}v}.
\end{align*}
Moreover, if $u=v$, then for $u \in [0,\pi_{3/2,3}/4)$,
\begin{align*}
\sin_{3/2,3}(2u)
&=\frac{\sin_{3/2,3}u (1+\cos_{3/2,3}^{3/2}u)}{\cos_{3/2,3}^{1/2}u (1+\sin_{3/2,3}^3u)},\\
\cos_{3/2,3}^{1/2}(2u)
&=\frac{\cos_{3/2,3}^{3/2}u-\sin_{3/2,3}^3u}{\cos_{3/2,3}^{3/2}u (1+\sin_{3/2,3}^3u)}.
\end{align*}
\end{lem}

\begin{proof}
The proof is based on the addition theorem of Dixon's elliptic functions. 
See Dixon \cite{Dixon1890}. 
\end{proof}

\begin{lem}[Cox-Shurman \cite{Cox2005}, Sato-Takeuchi \cite{Sato-Takeuchi2020}]
\label{lem:cox-shurman}
Let $x+y,\ x,\ y\in[0,\pi_{2,3})$ with $x \neq y$. Then,
\begin{equation}
s_{x+y}=\frac{2(s_x-s_y)((1-c_x)(1+c_y)-s_xs_y^2)}{(1-c_x)(1+c_y)^2-2s_x^2s_y(1+c_y)+s_xs_y^2(1+c_x)},
\label{eq:sx+y}
\end{equation}
where $s_z:=\sin_{2,3}{z}$ and $c_z:=\cos_{2,3}{z}$.
Moreover, if $x=y$, then for $x \in [0,\pi_{2,3}/2)$,
$$s_{2x}=\frac{4s_xc_x(3+c_x)^3}{(1+c_x)(8+s_x^3)^2}.$$
\end{lem}

\begin{proof}
Let $x,\ y,\ x+y \in[0,\pi_{2,3})$ with $x \neq y$. 
Then, $u:=x/2^{2/3},\ v:=y/2^{2/3}$ satisfies
$u,\ v,\ u+v \in[0,\pi_{3/2,3}/2)$ with $u \neq v$; hence, 
by Theorem \ref{thm:maf1} with $q=3$ and Lemma \ref{lem:dixon},
\begin{align*}
s_{x+y}
&=2^{2/3}S_{u+v}C_{u+v}^{1/2}\\
&=2^{2/3}\frac{(S_u^2C_v^{1/2}-C_u^{1/2}S_v^2)(S_uC_u^{1/2}-C_v^{1/2}S_v)}
{(S_uC_v-C_uS_v)^2},
\end{align*}
where $S_z:=\sin_{3/2,3}{z}$ and $C_z:=\cos_{3/2,3}{z}$.
Here, from \eqref{eq:mafcos} in Theorem \ref{thm:maf1},
\begin{align*}
S_z=\left(\frac{1-c_z}{2}\right)^{1/3},\quad
C_z=\left(\frac{1+c_z}{2}\right)^{2/3}.
\end{align*}
Substituting them, we obtain
\begin{align*}
s_{x+y}
&=2(s_x-s_y)\cdot\frac{(1-c_x)^{2/3}(1+c_y)^{1/3}-(1+c_x)^{1/3}(1-c_y)^{2/3}}{\left((1-c_x)^{1/3}(1+c_y)^{2/3}-(1+c_x)^{2/3}(1-c_y)^{1/3}\right)^2}.
\end{align*}
Multiplied the numerator and the denominator by $(1-c_x)^{1/3}(1+c_y)^{2/3}$, 
the fraction part of the right-hand side is equal to
$$\frac{(1-c_x)(1+c_y)-s_xs_y^2}{(1-c_x)(1+c_y)^2-2s_x^2s_y(1+c_y)+s_xs_y^2(1+c_x)},$$
which is the conclusion. Case $x=y$ is similarly proved.
\end{proof}

\begin{rem}
It follows from multiplying the numerator and the denominator of 
\eqref{eq:sx+y}
by $(1-c_x)(1+c_y)$ that
the right-hand side of 
\eqref{eq:sx+y} is symmetric with respect to $x$ and $y$. 
\end{rem}

\begin{thm}
\label{thm:3/2,2}
For $x\in[0,\pi_{3/2,2}/4)$,
\begin{equation}
\label{eq:3/2,2}
\sin_{3/2,2}(2x)=(\Phi \circ \Psi \circ \Phi^{-1})(\sin_{3/2,2}x),
\end{equation}
where
\begin{gather*}
\Phi(x)=\sqrt{1-\left(\frac{2-2x^2+2\sqrt{1-x^3}}{2+x^2+2\sqrt{1-x^3}}\right)^3},\\
\Psi(x)=\frac{4x \sqrt{1-x^3}(3+\sqrt{1-x^3})^3}{(1+\sqrt{1-x^3})(8+x^3)^2},\\
\Phi^{-1}(x)=\frac{6x-2(1-(1-x^2)^{1/3})^2}{(2+(1-x^2)^{1/3})^2}.
\end{gather*}
\end{thm}
\begin{proof}
Recall that for $z \in [0,2]$, 
\begin{gather*}
q\pi_{p,q}=p^*\pi_{q^*,p^*},\\
\sin_{p,q}\left(\frac{\pi_{p,q}}{2}z\right)=\cos_{q^*,p^*}^{q^*-1}\left(\frac{\pi_{q^*,p^*}}{2}(1-z)\right);
\end{gather*}
see \cite{Edmunds2012} and \cite{Kobayashi-Takeuchi}.
Let $x\in[0,\pi_{3/2,2}/4)=[0,3\pi_{2,3}/8)$.
Then, since $4x/\pi_{3/2,2} \in[0,1)$,
\begin{align*}
\sin_{3/2,2}(2x)
&=\cos_{2,3}\left(\frac{\pi_{2,3}}{2}-\frac{4}{3}x\right)
=\sqrt{1-\sin_{2,3}^3\left(\frac{\pi_{2,3}}{2}-\frac{4}{3}x\right)}.
\end{align*}
Applying \eqref{eq:sx+y}
in Lemma \ref{lem:cox-shurman},
we obtain
\begin{align*}
\sin_{2,3}\left(\frac{\pi_{2,3}}{2}-\frac{4}{3}x\right)
&=\sin_{2,3}\left(\frac{\pi_{2,3}}{2}+\frac{4}{3}x\right)\\
&=\frac{2(1-\sin_{2,3}(4x/3))\left(1+\cos_{2,3}(4x/3)-\sin^2_{2,3}(4x/3)\right)}{(1+\cos_{2,3}(4x/3)-\sin_{2,3}(4x/3))^2};
\end{align*}
hence
$$\sin_{3/2,2}(2x)=\sqrt{1-\left(\frac{2(1-\sin_{2,3}(4x/3))\left(1+\cos_{2,3}(4x/3)-\sin^2_{2,3}(4x/3)\right)}{\left(1+\cos_{2,3}(4x/3)-\sin_{2,3}(4x/3)\right)^2}\right)^3}.$$

Let $f(x):=\sin_{3/2,2}x$ and $g(x):=\sin_{2,3}(4x/3)$. Then,
\begin{align}
f(2x)
&=\sqrt{1-\left(\frac{2(1-g(x))\left(1+\sqrt{1-g(x)^3}-g(x)^2\right)}{\left(1+\sqrt{1-g(x)^3}-g(x)\right)^2}\right)^3}\nonumber\\
&=\sqrt{1-\left(\frac{2+2\sqrt{1-g(x)^3}-2g(x)^2}{2+2\sqrt{1-g(x)^3}+g(x)^2}\right)^3}
\nonumber \\
&=\Phi(g(x)). \label{eq:Phi}
\end{align}
Therefore, it is easy to see that
\begin{equation}
\label{eq:Phii}
g(x)
=\Phi^{-1}(f(2x))
=\frac{6f(2x)-2(1-\sqrt[3]{1-f(2x)^2})^2}{(2+\sqrt[3]{1-f(2x)^2})^2}.
\end{equation}

On the other hand, by the double-angle formula in Lemma \ref{lem:cox-shurman},
we see that $g(x)$ satisfies
\begin{equation}
\label{eq:Psi}
g(x)=\frac{4g(x/2)\sqrt{1-g(x/2)^3}(3+\sqrt{1-g(x/2)^3})^3}{(1+\sqrt{1-g(x/2)^3})(8+g(x/2)^3)^2}=\Psi \left(g\left(\frac{x}{2}\right)\right).
\end{equation}
Thus, using \eqref{eq:Phi}, \eqref{eq:Psi} and \eqref{eq:Phii} in this order, we obain 
$$f(2x)=\Phi(g(x))=\Phi \left(\Psi \left(g\left(\frac{x}{2}\right)\right) \right)=\Phi(\Psi(\Phi^{-1}(f(x)))),$$
which is the desired conclusion.
\end{proof}

\begin{rem}
Since the left-hand side of \eqref{eq:3/2,2}
converges to $1$ as $x \to \pi_{3/2,2}/4-0$,
we obtain 
$$\sin_{3/2,2}{\frac{\pi_{3/2,2}}{4}}
=\sqrt{135+78\sqrt{3}-6\sqrt{6(168+97\sqrt{3})}}
\approx 0.834896.$$ 
\end{rem}

\section{Generalization of Tangent Function} 
\label{sec:4}

Let $x \in [0,\pi_{p,q}/2)$ for simplicity.
Generalized tangent function is usually defined as
$$\tan_{p,q}{x}:=\frac{\sin_{p,q}{x}}{\cos_{p,q}{x}}.$$
Now, we propose another generalization of the tangent function:
$$\tau_{p,q}(x)=\frac{\sin_{p,q}{x}}{\cos_{p,q}^{p/q}{x}},$$
which coincides with $\tan_{p,q}{x}$ if $p=q$.
Then, $\tau_{p,q}(x)$ has properties very similar to $\tan{x}$ more than $\tan_{p,q}{x}$:
\begin{gather*}
1+\tau_{p,q}^{q}(x)=\frac{1}{\cos_{p,q}^p{x}},\\
\frac{d}{dx}(\tau_{p,q}(x))=\frac{1}{\cos_{p,q}^{p/q+p-1}{x}}=(1+\tau_{p,q}^{q}(x))^{1/q+1-1/p},\\
\tau_{2,q}(2^{2/q}x)=\frac{2^{2/q}\tau_{q^*,q}(x)}{(1-\tau_{q^*,q}^q(x))^{2/q}}.
\end{gather*}
Moreover, Lemma \ref{lem:sinhtosin} gives
\begin{equation}
\label{eq:reptau}
\sinh_{p,q}{x}=\tau_{r,q}(x), \quad \cosh_{p,q}{x}=(1+\tau_{r,q}^q(x))^{1/p},
\end{equation}
where $r=pq/(pq+p-q)$.
In particular, Lemma \ref{lem:sintosinh} and \eqref{eq:reptau} with \eqref{eq:r=p}
yield
$$\tan_{p,q}{x}=\tau_{p,q}(x)(1+\tau_{p,q}^q(x))^{1/p-1/q}.$$

Let us apply this function to integration. 
Setting $\tau_{q^*,q}(x/2^{2/q})=t$, we have
\begin{gather*}
\sin_{2,q}{x}=\frac{2^{2/q}\sin_{q^*,q}{(x/2^{2/q})}\cos_{q^*,q}^{q^*-1}{(x/2^{2/q})}}
{(\cos_{q^*,q}^{q^*}{(x/2^{2/q})}+\sin_{q^*,q}^{q}{(x/2^{2/q}))^{2/q}}}
=\frac{2^{2/q}t}{(1+t^q)^{2/q}},\\
\cos_{2,q}{x}=\frac{\cos_{q^*,q}^{q^*}{(x/2^{2/q})}-\sin_{q^*,q}^q{(x/2^{2/q})}}{\cos_{q^*,q}^{q^*}{(x/2^{2/q})}+\sin_{q^*,q}^q{(x/2^{2/q})}}
=\frac{1-t^q}{1+t^q},\\
\tan_{2,q}{x}=\frac{\sin_{2,q}{x}}{\cos_{2,q}{x}}=\frac{2^{2/q}t}{(1+t^q)^{2/q-1}(1-t^q)},\\
\tau_{2,q}{(x)}=\frac{2^{2/q}\tau_{q^*,q}{(x/2^{2/p})}}{(1-\tau_{q^*,q}{(x/2^{2/q})})^{2/q}}=\frac{2^{2/q}t}{(1-t^q)^{2/q}},\\
\frac{dt}{dx}=\frac{1}{2^{2/q}\cos_{q^*,q}^{2/(q-1)}{(x/2^{2/p})}}
=\frac{(1+t^q)^{2/q}}{2^{2/q}}.
\end{gather*}
For example, 
\begin{align*}
\int \frac{dx}{\sin_{2,q}{x}}
&=\int \frac{dt}{t}=\log{|t|}+C=\log{|\tau_{q^*,q}{(2^{2/q}x)}|}+C,\\
\int \frac{dx}{\cos_{2,q}^{2/q}{x}}
&=2^{2/q}\int \frac{dt}{(1-t^q)^{2/q}}
=2^{2/q}\sin_{q/2,q}^{-1}{t}+C\\
&=2^{2/q}\sin_{q/2,q}^{-1}{(\tau_{q^*,q}{(2^{2/q}x)})}+C.
\end{align*}




\end{document}